\documentclass[11pt]{amsart}

\title[Polynomial convexity of $\overline\partial$-flat perturbations of totally real sets]{Polynomial convexity of $\overline\partial$-flat perturbations of totally real sets}

\author{Leandro Arosio$^*$ \& H{\aa}kan Samuelsson Kalm \& Erlend F.\ Wold}
\subjclass[2010]{}
\address{Leandro Arosio, Dipartimento Di Matematica, Universit\`{a} di Roma ``Tor Vergata'',
Via Della Ricerca Scientifica 1, 00133 Roma, Italy}
\email{arosio@mat.uniroma2.it}
\address{H{\aa}kan Samuelsson Kalm, Department of Mathematical Sciences, Division of Algebra and Geometry, University of Gothenburg and 
Chalmers University of Technology, SE-412 96 G\"{o}teborg, Sweden}
\email{hasam@chalmers.se}
\address{Erlend F. Wold, Department of Mathematics, University of Oslo, PO-BOX 1053,
Blindern, 0316 Oslo, Norway}
\email{erlendfw@math.uio.no}

\date{\today}

\usepackage[english]{babel}
\usepackage{amsmath,calligra}
\usepackage{amsthm,amssymb,latexsym}
\usepackage[all,cmtip]{xy}
\usepackage{url,ae}
\usepackage{t1enc}
\usepackage{bm}
\usepackage{mathrsfs}
\usepackage{graphicx}
\usepackage{geometry}
\usepackage{pstricks,pst-plot}
\usepackage{tikz-cd}

\newtheorem{proposition}{Proposition}[section]
\newtheorem{theorem}[proposition]{Theorem}
\newtheorem{lemma}[proposition]{Lemma}
\newtheorem{corollary}[proposition]{Corollary}

\theoremstyle{definition}

\newtheorem{remark}[proposition]{Remark}

\numberwithin{equation}{section}

\DeclareMathOperator{\Hom}{\mathscr{H}\text{\kern -3pt {\calligra\Large om}}\,}
\DeclareMathOperator{\Ext}{\mathscr{E}\text{\kern -3pt {\calligra\Large xt}}\,\,}
\DeclareMathOperator{\Image}{\mathscr{I}\text{\kern -3pt {\calligra\Large m}}\,}
\DeclareMathOperator{\Ker}{\mathscr{K}\text{\kern -3pt {\calligra\Large er}}\,}

\newcommand{\PM}{\mathscr{P} \kern -3pt \mathscr{M}}

\newcommand{\CH}{\mathscr{C} \kern -2pt \mathscr{H}}

\def\newop#1{\expandafter\def\csname #1\endcsname{\mathop{\rm #1}\nolimits}}
\newop{span}

\begin{document}
\nocite{*}
\bibliographystyle{plain}

\begin{abstract}
We show that if $X$ is a totally real $d$-dimensional manifold 
attached to a polynomially convex
compact set $K$ in $\mathbb{C}^n$, $d<n$, then there are arbitrarily small
perturbations $X'$ of $X$ such that $K\cup X'$ is
polynomially convex. The perturbations are induced by diffeomorphisms of 
$\mathbb{C}^n$ fixing $K$, which are $\bar\partial$-flat on $K\cup X$,  and which are arbitrarily $C^k$-close to the identity.

\end{abstract}
	
\thanks{${}^*$  Partially supported by   INdAM, by  PRIN {\sl  Real and Complex Manifolds: Geometry and Holomorphic Dynamics} n. 2022AP8HZ9, and by the MUR Excellence Department Project MatMod@TOV
		CUP:E83C23000330006}

\maketitle
\thispagestyle{empty}

\section{Introduction}

The polynomially convex hull $\widehat K$ of a compact set $K\subset \mathbb{C}^n$
is the set of points $z\in\mathbb{C}^n$ such that $|p(z)|\leq\sup_K|p|$ for all
polynomials $p\in\mathbb{C}[z_1,\ldots,z_n]$. It is the maximal 
ideal space of the uniform algebra $[z_1,\ldots,z_n]_K$ of continuous functions on 
$K$ that are limits in $C^0(K)$ of polynomials. If $\widehat K=K$, then $K$ is said to be
polynomially convex. The celebrated Oka--Weil theorem says that any function
holomorphic in a neighborhood of a polynomially convex compact set $K$
can be uniformly approximated on $K$ by polynomials.

Another important notion in complex analysis is that of a totally real set. 
A submanifold $X\subset\mathbb{C}^n$ is totally real if for all $z\in X$,
the tangent space $T_zX$
contains no non-trivial complex subspace. This notion is connected to,
and has a subtle interplay with polynomial
convexity. 
The objective of this note is to prove the following generalization of
\cite[Theorem~1.4]{AW}.

\begin{theorem}\label{glass}
Let $K\subset \mathbb C^n$ be a polynomially convex compact set, and let $X\subset\mathbb C^n\setminus K$
be a closed, bounded,  totally real $C^\infty$-smooth 
submanifold with $\mathrm{dim}(X)< n$.
%, of class $C^k, k\in\mathbb N\cup\{\infty\}$. 
Then for any $k\in\mathbb{N}$, %finite $m\leq k$, 
%$\delta\in C(\mathbb C^n), \delta>0$,
$\delta>0$, and $\kappa\in\mathbb{N}$
there exists a $C^\infty$-smooth diffeomorphism 
$F\colon\mathbb{C}^n\to\mathbb{C}^n$ %\in\mathrm{Diff}_k(\mathbb C^n)$
such that the following hold: 
\begin{itemize}
\item[(i)] $\|F-\mathrm{id}\|_{C^k(\mathbb{C}^n)}<\delta$, 
\item[(ii)] $F=\mathrm{id}$ on $K$
\item[(iii)] 
$\displaystyle{\frac{\partial^{\alpha+\beta}F}{\partial z^\alpha\bar{z}^\beta}(z)}=0$ 
for all $z\in K\cup X$ if $|\alpha|+|\beta|\leq\kappa$ and $|\beta|\geq 1$,
%$\overline\partial$-flat to order $m$ along $X\cup K$
\item[(iv)] $F(K\cup X)$ is polynomially convex.
%\item[(iv)] $F(K\cup X)$ has bounded exhaustion hulls in $\mathbb C^n$. 
\end{itemize}
\end{theorem}

The main novelty compared to \cite[Theorem~1.4]{AW} it that the perturbation of
 $X$ is induced by a
diffeomorphism $F$ of $\mathbb{C}^n$ satisfying (iii). 
In particular, $DF$ is $\mathbb{C}$-linear at all $z\in K\cup X$ since $\bar\partial F=0$
there. Notice that $F(X)$ thus is totally real.

We think Theorem~\ref{glass} has independent interest in 
view of the study of %in 
%understanding 
the structure of polynomially convex hulls and finding embeddings of 
manifolds into $\mathbb{C}^n$ such that the hull of the image has certain properties,
see, e.g., \cite{alexander, DL, F, FR, Gupta, GS1, GS2, ISW, IS, S, Stout, W}.
%....??? studying polynomially
%convex hulls; presence or absence of analytic structure; prescribed hulls 

The application we primarily have in mind is the following. %, see \cite{ASKW}. 
Let $M\subset\mathbb{C}^n$ be a compact submanifold of dimension
$d<n$ and let $K$ be the set of 
$z\in M$ such that $T_zM$ contains a non-trivial complex subspace.
Assume that $K$ is polynomially convex. Applying Theorem~\ref{glass} with
$X=M\setminus K$ it follows that there is an arbitrarily small perturbation $M'$ of
$M$ fixing $K$ such that $M'$ is polynomially convex and $M'\setminus K$ is totally real.
This is used in \cite{ASKW} to show that if $d=n-1\leq 11$, then any such
$M$ can be slightly perturbed to a submanifold $M'$ that is
polynomially convex and $[z_1,\ldots,z_n]_{M'}=C^0(M')$.
%such that any continuous function on $M'$ can be 
%uniformly approximated on $M'$ by polynomials in $\mathbb{C}^n$.

\section{Preliminaries}

We recall some results and notions 
from \cite{LW} and \cite{AW}. A closed set $X$ in $\mathbb{C}^n$ is a 
\emph{totally real set} of class $C^k$, $k\in\mathbb{N}\cup \{\infty\}$,
if for each $z\in X$ there is a neighborhood $U_z$ of $z$ in 
$\mathbb{C}^n$ and a $C^k$-smooth closed totally real submanifold $Y$ of $U_z$
such that $X\cap U_z\subset Y$. Possibly shrinking $U_z$ we may assume that
$Y$ is a $C^k$-embedding of the unit ball in some $\mathbb{R}^\ell$. 
If we fix a locally finite collection $\mathcal{Y}=\{Y_j\}$ of such $Y$ so that
$X\subset\cup_jY_j$ we say that the pair $(X,\mathcal{Y})$ is a
\emph{parametrized totally real set} of class $C^k$. 
A $C^k$-perturbation of  $(X,\mathcal{Y})$
is a set $\varphi(X)$, where $\varphi\colon\cup_jY_j\to\mathbb{C}^n$ is 
$C^k$-smooth on each $Y_j$.
For any compact set $K\subset\mathbb{C}^n$ we let
$$
h(K)=\overline{\widehat{K}\setminus K}.
$$
The following is \cite[Proposition~4]{LW}. 

\begin{proposition}\label{LW1}
Let $(X,\mathcal Y)$ be a compact parametrized totally real set in $\mathbb C^n$ of
class $C^1$, and 
let $U'\Subset U\Subset\mathbb C^n$ be open sets. Then there exists a neighborhood 
$\Omega$ of $X$ such that the following hold:
\begin{itemize}
\item[(1)] If $S\subset X$ is closed and $K\subset U'$ is compact, then 
\begin{equation}\label{cake}
\widehat{K\cup S}\subset U'\cup\Omega\Rightarrow 
%\overline{\widehat{K\cup S}\setminus (K\cup S)}
h(K\cup S)\subset U.
\end{equation}
\item[(2)] If $X'$ is a sufficiently small $C^1$-perturbation of $X$, then \eqref{cake} 
holds with $S'$ in place 
of $S$ for all closed subsets $S'$ of $X'$.
\end{itemize}
\end{proposition}

The following is a corollary (Corollary 2 in \cite{LW}).

\begin{corollary}\label{LW2}
Let $(X,\mathcal Y)$ be a compact parametrized totally real set in $\mathbb C^n$
of class $C^1$, let 
$K\subset\mathbb C^n$ be compact, and let $U\supset K$ be an open set such that 
\begin{equation}\label{cakes}
%\overline{\widehat{K\cup X}\setminus (K\cup X)}
h(K\cup X)\subset U.
\end{equation}
If $X'$ is  a sufficiently small $C^1$-perturbation of $X$, then \eqref{cakes} holds with
$X$ replaced by $X'$.
\end{corollary}

Suppose $F\colon \mathbb{C}^n\to\mathbb{C}^n$ is a diffeomorphism. 
Corollary~\ref{LW2} implies that there is a constant $c(X,K,U)>0$ such that if
\begin{equation}\label{report}
\|F-\text{id}\|_{C^1(\mathbb{C}^n)}< c(X,K,U),
\end{equation}
then \eqref{cakes} holds with $X$ replaced by $F(X)$.

The following was Proposition 6 in \cite{LW} in the case of $C^1$-regularity, and a special case of Lemma 3.1 in \cite{AW}.

\begin{proposition}\label{LW3}
Let $f$ be a $C^\infty$-smooth embedding of a neighborhood of the closed unit cube
$\overline{I}\subset\mathbb{R}^d$ into $\mathbb{C}^n$, $d<n$, such that
$f(\overline{I})$ is totally real. 
Let $K\subset\mathbb C^n$ be polynomially convex, and let $U$ be a neighborhood of $K$. Then 
for any $\epsilon>0$, $k\in\mathbb N$, there exists a $C^\infty$-smooth embedding $f_\epsilon\colon \overline{I}\rightarrow \mathbb C^n$
such that 
\begin{itemize}
\item[(1)] $\|f_\epsilon-f\|_{C^k(\overline{I})}<\epsilon$,
\item[(2)] $f_\epsilon=f$ near $f^{-1}(K)$, 
\item[(3)] $h(K\cup f_\epsilon(\overline{I}))\subset U$.
%$\overline{\widehat{K\cup f_\epsilon(\overline{I})}\setminus (K\cup %f_\epsilon(\overline{I}))}\subset U$.
\end{itemize}
\end{proposition}

\begin{remark}\label{rem1}
The meaning of the word \emph{near} in (2) is independent of $\epsilon$. This follows directly from the proof in \cite{AW} and \cite{LW}
but can also be seen as follows. Choose $K'$ polynomially convex with $K'\subset U$, and with $K\subset\mathrm{int}(K')$. 
Then $V=f^{-1}(\mathrm{int}(K'))$ is an open set containing the closed set $f^{-1}(K)$. Applying the proposition 
with $K'$ instead of $K$ gives $f_\epsilon=f$ on $V$.
\end{remark}

\section{Proof of Theorem~\ref{glass}}
To prove Theorem~\ref{glass} we use two lemmas, Lemma~\ref{easywhitney} 
and Lemma~\ref{glasss} below.
Let $X$ be a $d$-dimensional totally real submanifold of an open set in $\mathbb{C}^n$ 
and let $I\subset\mathbb{R}^d$ be the open unit cube. We say that
$f\colon\overline{I}\to X$ is a local parametrization if $f$ is a diffeomorphism 
from a neighborhood of $\overline{I}$ to an open submanifold of $X$. 
We have the following elementary lemma.

\begin{lemma}\label{easywhitney}
Let $O\subset\mathbb{C}^n$ be an open set, $X\subset O$ a totally real smooth 
closed submanifold of dimension $d$, and $f\colon\overline{I}\to X$ 
a local parametrization.
%Let $I\subset\mathbb{R}^d$ be the open unit cube and suppose $f\colon I\to X$ is local %smooth totally real parametrization of $X$.
There is a neighborhood $O'$ of $X$ in $O$ with the following properties.
For any $\kappa,\ell\in\mathbb{N}$ there is a constant $C>0$ such that if 
$\eta\in C^\infty_0(f(I),\mathbb{C}^m)$,
then there is a smooth extension $\widetilde{\eta}$ of $\eta$ to $O'$ and
\begin{itemize}
\item[(i)] $\displaystyle{\frac{\partial^{\alpha+\beta}\widetilde{\eta}}{\partial z^\alpha\bar{z}^\beta}(z)}=0$ 
for all $z\in X$ if $|\alpha|+|\beta|\leq\kappa$ and $|\beta|\geq 1$,
%The differential of $\widetilde{h}$ is $\mathbb{C}$-linear at all $x\in X$.
\item[(ii)] 
$\|\widetilde{\eta}\|_{C^\ell(O')}\leq C\|\eta\|_{C^{\ell+\kappa}(f(\overline{I}))}$,
\item[(iii)] $\widetilde{\eta}=0$ in a neighborhood of $X\setminus f(I)$,
\item[(iv)] If $K\subset\mathbb{C}^n$ is compact and $\eta=0$ in a neighborhood of 
$K\cap f(\overline{I})$ in $ f(\overline{I})$, then $\widetilde{\eta}=0$ in a neighborhood
only depending on $\text{supp}\,\eta$
of $K\cap X$ in $\mathbb{C}^n$.
\end{itemize}
\end{lemma}

\begin{proof}
Assume first that $d=n$. Define $\hat{f}\colon I+i\mathbb{R}^n\to\mathbb{C}^n$ by
\begin{equation}\label{friday}
\hat{f}(x+iy)=\sum_{|\alpha|\leq \kappa}
\frac{\partial^\alpha f}{\partial x^\alpha}(x)\cdot \frac{(iy)^\alpha}{\alpha !}.%+iDf_x(y).
\end{equation}
Clearly, $\hat{f}$ is an extension of $f$ to $I+i\mathbb{R}^n$ and one checks that
if $\zeta=x+iy$ and $|\beta|\geq 1$, then 
%if $u+iv$ is a tangent vector of $I+i\mathbb{R}^n$ at $x\in I$, then
\begin{equation}\label{monday}
\frac{\partial^{\alpha+\beta} \hat f}{\partial \zeta^\alpha\partial\bar{\zeta}^\beta}(x+iy)
=\mathcal{O}(|y|^{\kappa-|\alpha|-|\beta|+1}).
%D\hat{f}_x(u+iv)=Df_x(u)+iDf_x(v).
\end{equation}
In particular, the differential of $\hat{f}$ is $\mathbb{C}$-linear at all $x\in I$. 
Since $Df_x$ has real rank $n$ and $X$ is totally real it follows that
$D\hat{f}_x$ has complex rank $n$. Thus $\hat{f}$ is a diffeomorphism from a
neighborhood $V$ of $I$ in $I+i\mathbb{R}^n$ to a neighborhood $\hat f(V)$ of $f(I)$ in 
$\mathbb{C}^n$. We take $V$ to be of the form $I+i(-\delta,\delta)^n$ for a sufficiently 
small $\delta>0$.

Let $\eta\in C^\infty_0(f(I),\mathbb{C}^m)$.
Define $\hat{\eta}\colon I+i\mathbb{R}^n\to\mathbb{C}^n$
as in \eqref{friday} with $f$ replaced by $\eta\circ f$. Then $\hat \eta$ is an extension of 
$\eta\circ f$ to $I+i\mathbb{R}^n$
and satisfies \eqref{monday} with $\hat f$ replaced by $\hat\eta$. 
%$D\hat h_x$ is $\mathbb{C}$-linear for all $x\in I$.
Notice that if $\eta\circ f=0$ in an open set $\omega\subset I$, then 
$\hat \eta=0$ in $\omega+i\mathbb{R}^n$. Since $\eta$ has compact support in $f(I)$,
and if $K\subset\mathbb{C}^n$ is compact and $\eta=0$ in a neighborhood of 
$K\cap f(\overline{I})$ in $ f(\overline{I})$, it follows that $\hat \eta=0$ in a neighborhood
of $(\partial I\cup f^{-1}(K))+i\mathbb{R}^n$.
Now,
$$
\widetilde{\eta}:=\hat \eta\circ \hat f^{-1}
$$
is an extension of $\eta$ to $\hat f(V)$, 
it satisfies (i)
%$$
%\frac{\partial^{\alpha+\beta}\widetilde{\eta}}{\partial z^\alpha\bar{z}^\beta}(z)=0
%$$
%$D\widetilde{h}$ is $\mathbb{C}$-linear 
with $X$ replaced by $f(I)$, and 
$$
\|\widetilde \eta\|_{C^\ell(\hat f(V))}\leq C \|\eta\|_{C^{\ell+\kappa}(f(\overline{I}))}
$$
for some constant $C>0$.
If $\eta=0$ in a neighborhood of 
$K\cap f(\overline{I})$ in $ f(\overline{I})$, then possibly after shrinking $V$ we can assume
that $\widetilde{\eta}=0$ in a neighborhood of $K\cap \hat f(V)$ in $\hat f(V)$.

Let $V'$ be a neighborhood of 
$X\setminus f(I)$ such that
$V'\cap \hat f(V)\subset\{\widetilde{\eta}=0\}$. We extend $\widetilde{\eta}$
to $O'=V'\cup \hat f(V)$ by setting $\widetilde \eta=0$ in $V'$. It follows that 
$\widetilde \eta$ has the required properties.

\smallskip

Now suppose that $d<n$. As in the first part of the proof one obtains an 
extension $\check{\eta}$ of
$\eta$ to a $2d$-dimensional submanifold $Y\subset O$ containing $f(I)$ such that
for all $p\in f(I)$,
$T_pY\simeq\mathbb{C}^d$ and $D\check{\eta}_p$ is $\mathbb{C}$-linear
to order $\kappa$.
Let $\tau\colon \widetilde{V}\to Y$ be a tubular neighborhood of $Y$. Then 
$\widetilde{\eta}=\tau^* \check{\eta}$  is an extension of $\eta$ 
to $\widetilde{V}$ such that (i) holds with $X$ replaced by $f(I)$.
%$D\widetilde{h}_p$ is $\mathbb{C}$-linear for all $p\in f(I)$. 
One now concludes the proof in the same way as in the case $d=n$.
\end{proof}

\begin{lemma}\label{glasss}
Let $K\subset \mathbb{C}^n$ be a polynomially convex compact set, and let $X$
be a  bounded closed totally real $C^\infty$-smooth submanifold 
of $\mathbb{C}^n\setminus K$ with $\mathrm{dim}(X)=d<n$.
% and of class $C^s, s\in\mathbb N\cup\{\infty\}$. For $0<r_1<r_2$, we let 
%$X(r_j)=X\cap \overline{r_j\mathbb B^n}$.
Let $U$ be a neighborhood of $K$. 
Then for any $\epsilon>0$ and $\kappa, \ell\in\mathbb{N}$ 
%k\leq s$, there exists a $C^1$-smooth 
there exists a $C^\infty$-smooth 
diffeomorphism $F_\epsilon\colon\mathbb{C}^n\rightarrow\mathbb{C}^n$
such that the following hold: 
\begin{itemize}
\item[(i)] $\|F_\epsilon-\mathrm{id}\|_{C^\ell(\mathbb C^n)}<\epsilon$,
\item[(ii)] $\displaystyle{\frac{\partial^{\alpha+\beta}F_\epsilon}{\partial z^\alpha
\partial\bar{z}^\beta}(z)=0}$
%$DF_\epsilon(x)$ is $\mathbb{C}$-linear 
for all $z\in X$ if $|\alpha|+|\beta|\leq\kappa$ and $|\beta|\geq 1$, 
\item[(iii)] $F_\epsilon = \mathrm{id}$ in a neighborhood of $K$,
\item[(iv)] $\widehat{F_\epsilon(K\cup X)}\subset U\cup F_\epsilon(X)$.
\end{itemize}
\end{lemma}

\begin{proof}
Fix a polynomially convex compact set $K'$ such that 
\begin{equation}\label{home}
K\subset\mathrm{int}(K')\subset K'\subset  U.
\end{equation}
Let $I\subset\mathbb{R}^d$ be the open unit cube and $I_0\Subset I$ a 
slightly smaller cube. Fix local %$C^\infty$-smooth totally real 
parametrizations 
$g_j\colon \overline{I}\rightarrow X$, $j=1,\ldots,m$, such that 
$\{g_j(I_0)\}_{j=1}^m$ covers $\overline{X\setminus K'}$ and
$g_j(\overline{I})\cap K =\emptyset$ for $j=1,\ldots,m$.
We let $X^j=g_j(\overline{I}_0)$, and for $1\leq i\leq m$ we let $X_i=\cup_{j\leq i} X^j$. 
Moreover, take open sets $U_i$ such that $K'\subset U_1\Subset\cdots\Subset U_m=U$. 
By induction on $i$ 
we will construct diffeomorphisms $F_{\epsilon,i}$ as in the lemma, but where (iii) is
replaced by 
\begin{equation}\label{three}
F_{\epsilon,i}=\text{id}\,\,\text{on}\,\,K'
\end{equation}
and (iv) is replaced by 
\begin{equation}\label{four}
\widehat{F_{\epsilon,i}(K'\cup X_i)}\subset  U_i\cup F_{\epsilon,i}(X_i).
\end{equation}
%and in addition,
%\begin{equation}\label{floorball1}
%F_{\epsilon,i}=\text{id}\,\,\text{on}\,\, K_{i-1}:=\widehat{F_{\epsilon,i-1}(K'\cup X_{i-1})},
%\end{equation}
%where we let $K_0=K'$. 
Then $F_{\epsilon,m}$ has the properties in the lemma by \eqref{home} and since
$K\cup X\subset K'\cup X_m$.

To start the induction, we  take $F_{\epsilon,0}=\text{id}$, which has the required properties for $i=0$ if we let $X_0=\emptyset$ and $K'\subset U_0\Subset U_1$.
Assume now that we have $F_{\epsilon,i}$ for some $0\leq i\leq m-1$.
Let $\Omega$ be a neighborhood of $F_{\epsilon,i}(X_i\cup g_{i+1}(\overline{I}))$ as in
Proposition~\ref{LW1} with data $K'\subset U_i\Subset U_{i+1}$.
%such that
%\begin{multline*}
%[K'\cup F_{\epsilon,i}(X_i\cup g_{i+1}(\overline{I}))]^{\widehat{}}\subset
%U_i\cup\Omega \\
%\Rightarrow
%[K'\cup F_{\epsilon,i}(X_i\cup g_{i+1}(\overline{I}))]^{\widehat{}}\subset
%U_{i+1}\cup F_{\epsilon,i}(X_i\cup g_{i+1}(\overline{I})),
%\end{multline*}
%cf.\ Proposition~\ref{LW1}. 
For notational convenience, let
\begin{equation}\label{clouds}
K_i=\widehat{F_{\epsilon,i}(K'\cup X_i)},
\end{equation}
let $f=F_{\epsilon,i}\circ g_{i+1}$, and notice that
$K_i\cup f(\overline{I})\subset U_i\cup\Omega$ in view of \eqref{four}.
%let 
%$\widetilde{X}=f(\overline{I})$. 
Replacing $U$ in Proposition~\ref{LW3} by
$U_i\cup\Omega$ it follows that there is a sequence 
$f_{\epsilon_j}\in C^\infty(\overline{I},\mathbb{C}^n)$ such that 
$f_{\epsilon_j}\to f$, in $C^{\ell+\kappa}(\overline{I})$ as $ j\to\infty$, 
$f_{\epsilon_j}=f$ in a neighborhood of $f^{-1}(K_i)$, and
\begin{equation}\label{cloudss}
\widehat{K_i\cup f_{\epsilon_j}(\overline{I})}\subset U_i\cup\Omega.
\end{equation}
The neighborhood of $f^{-1}(K_i)$ where $f_{\epsilon_j}=f$ is independent of $j$, 
see Remark~\ref{rem1}.
Take $\chi_1\in C^\infty_0(I)$ such that $\chi_1=1$ in a neighborhood of $\overline{I}_0$
and let $k_j=f+\chi_1\cdot (f_{\epsilon_j}-f)$. Then
$k_j\to f$ in $C^{\ell+\kappa}(\overline{I})$
as $j\to\infty$, $k_j=f$ in a neighborhood of $f^{-1}(K_i)\cup\partial I$ 
independent of $j$, and
\begin{equation}\label{cloudsss}
k_j=f_{\epsilon_j}\,\,\text{in a neighborhood of}\,\, \overline{I}_0.
\end{equation}
Letting $\eta_j=k_j\circ f^{-1}-\text{id}$ we then have 
$\eta_j\in C^\infty_0(f(I),\mathbb{C}^n)$, $\eta_j\to 0$ in 
$C^{\ell+\kappa}(f(\overline{I}),\mathbb{C}^n)$ as $j\to\infty$, and 
$\eta_j=0$
in a neighborhood of $K_i\cap f(\overline{I})$ in $f(\overline{I})$ 
independent of $j$.

By 
Lemma~\ref{easywhitney} there are $C^\infty$-smooth extensions $\tilde{\eta}_j$ of
$\eta_j$ to a neighborhood $O'$ of $F_{\epsilon,i}(X)$ in $\mathbb{C}^n\setminus K$
such that %the differential of $\tilde{h}_j$ is $\mathbb{C}$-linear at all 
%$x\in F_{\epsilon,i}(X)$
$\tilde{\eta}_j$ satisfy (i) of Lemma~\ref{easywhitney} with $X$ replaced by 
$F_{\epsilon,i}(X)$,
$\tilde{\eta}_j\to 0$ in $C^\ell(O')$ as $j\to\infty$,
and $\tilde{\eta}_j=0$ in a neighborhood of 
$(K_i\cap F_{\epsilon,i}(X))\cup(F_{\epsilon,i}(X)\setminus f(I))$
independent of $j$. Take $\chi_2\in C^\infty_0(O')$ such that
$\chi_2=1$ in a neighborhood of $f(\overline{I})$ and
$
\text{supp}\, \chi_2\cap K_i\subset \{\tilde{\eta}_j=0\}.
$
Now let $F_j=\text{id}+\chi_2\tilde{\eta}_j$. Then
\begin{equation}\label{goal}
F_j\to\text{id}\,\,\text{in}\,\, C^\ell(\mathbb{C}^n,\mathbb{C}^n),
\end{equation}
\begin{equation}\label{floorball2}
F_j=\text{id}\,\,\text{on}\,\, K_i\cup (F_{\epsilon,i}(X)\setminus f(I)),
\end{equation}
\begin{equation}\label{gooal}
\frac{\partial^{\alpha+\beta}F_j}{\partial z^\alpha\partial\bar{z}^\beta}(z)=0,\,\,
\forall z\in F_{\epsilon,i}(X)\,\, \text{if}\,\, |\alpha|+|\beta|\leq \kappa, \, |\beta|\geq 1.
%DF_j(x)\,\, \text{is}\,\, \mathbb{C}\text{-linear}\,\, \forall x\in F_{\epsilon,i}(X).
\end{equation}
Moreover, 
\begin{equation}\label{floorball3}
[F_j\circ F_{\epsilon,i}(K'\cup X_{i+1})]^{\widehat{}} 
\subset U_i\cup\Omega.
\end{equation}
In fact, in view of \eqref{clouds}, \eqref{floorball2}, \eqref{cloudsss}, and \eqref{cloudss}
we have
\begin{eqnarray*}%\label{floorball3}
[F_j\circ F_{\epsilon,i}(K'\cup X_{i+1})]^{\widehat{}} 
& \subset &
[F_j(K_i\cup F_{\epsilon,i}(X^{i+1}))]^{\widehat{}}
=
[K_i\cup F_j\circ F_{\epsilon,i}(X^{i+1})]^{\widehat{}}\\
&=&
[K_i\cup(\text{id}+\eta_j)\circ f(\overline{I}_0)]^{\widehat{}}
=
[K_i\cup k_j(\overline{I}_0)]^{\widehat{}}\\
&=&
[K_i\cup f_{\epsilon_j}(\overline{I}_0)]^{\widehat{}}
\subset 
U_i\cup\Omega,
\end{eqnarray*}
which shows \eqref{floorball3}.
By \eqref{three} %, \eqref{clouds} 
and \eqref{floorball2}, both
$F_{\epsilon,i}$ and $F_j$ are the identity on $K'$. Thus
\begin{equation}\label{floorball4}
F_j\circ F_{\epsilon,i}(K'\cup X_{i+1})=K'\cup F_j\circ F_{\epsilon,i}(X_{i+1}).
\end{equation}
Since $F_j\circ F_{\epsilon,i}(X_{i}\cup g_{i+1}(\overline{I}))$ is a small perturbation of 
$F_{\epsilon,i}(X_{i}\cup g_{i+1}(\overline{I}))$ and 
$F_j\circ F_{\epsilon,i}(X_{i+1})$ is a closed set contained in $F_j\circ F_{\epsilon,i}(X_{i}\cup g_{i+1}(\overline{I}))$ it follows by Proposition~\ref{LW1}~(2), \eqref{floorball3}, 
\eqref{floorball4},
and the choice of $\Omega$ that
$$
[F_j\circ F_{\epsilon,i}(K'\cup X_{i+1})]^{\widehat{}} 
\subset U_{i+1}\cup F_j\circ F_{\epsilon,i}(X_{i+1})
$$
if $j$ is large enough. Letting $F_{\epsilon,i+1}=F_j\circ F_{\epsilon,i}$ for $j$ sufficiently
large thus \eqref{four} holds for $i+1$. By \eqref{goal} and \eqref{gooal},
$F_{\epsilon,i+1}$ satisfies (i) and (ii) of the lemma. Finally,  
\eqref{three} holds for
$i+1$ in view of \eqref{floorball2}.  
This completes the induction step.
\end{proof}

\begin{proof}[Proof of Theorem~\ref{glass}]
The proof of Theorem~\ref{glass} is concluded using Lemma~\ref{glasss} in 
a similar way
as \cite[Theorem~1.4]{AW} is concluded using \cite[Lemma~3.1]{AW}. 

Let $U_1\supset U_2\supset\cdots$ be neighborhoods of $K$ such that
$\cap_1^\infty U_j=K$. Let $K_j\subset U_j$ be compact polynomially convex sets
such that
$K\subset\text{int}(K_j)$. Let $\varphi\colon\mathbb{N}\to\mathbb{N}$ be increasing such that $U_{\varphi(j)}\Subset \text{int}(K_j)$, let $d_j>0$ be the distance between 
$U_{\varphi(j)}$ and $K_j^c$, and let $d_0=\delta$.

Let $F_0=\text{id}$. Inductively, using Lemma~\ref{glasss}
with data $K_j$, $X_{j-1}:=F_{j-1}\circ\cdots\circ F_0(X)$, and $U_{j}$,
there are diffeomorphisms $F_{j}\colon\mathbb{C}^n\to\mathbb{C}^n$, $j=1,2,\cdots$, 
such that
$F_{j}=\text{id}$ in a neighborhood of $K_j$, 
$F_j$ satisfy (ii) of Lemma~\ref{glasss} with $X$ replaced by $X_{j-1}$, 
%$DF_{j}$ is $\mathbb{C}$-linear at all $x\in X_{j-1}$, 
\begin{equation}\label{shower}
\big[F_{j}(K_j\cup X_{j-1})\big]^{\widehat{}}\subset U_{j}\cup F_{j}(X_{j-1}),
\end{equation}
$\|F_{1}-\text{id}\|_{C^{k}(\mathbb{C}^n)}<d_0/3$, and
$$
\|F_{j}-\text{id}\|_{C^{k+j-1}(\mathbb{C}^n)}<
\min
\left\{\frac{d_{0}}{3^{j}},\frac{c_1}{3^j},\frac{d_1}{3^{j-1}},\frac{c_2}{3^{j-1}},
\ldots,\frac{d_{j-1}}{3},\frac{c_j}{3}\right\},
$$
where $c_j=c(X_{j-1}\setminus U_{\varphi(j-1)},K_j,U_j)$; see \eqref{report}.
Let $F_j^\ell=F_\ell\circ\cdots\circ F_j$. Then $F_j^\ell$ converge to smooth mappings
$F^\infty_j\colon\mathbb{C}^n\to\mathbb{C}^n$ as $\ell\to\infty$ 
such that $F^\infty_j=\text{id}$
on $K$, $F^\infty_j$ satisfy (iii) of the theorem with $X$ replaced by $X_{j-1}$,
%$DF^\infty_j$ is $\mathbb{C}$-linear at all $x\in K\cup X_{j-1}$, 
and
$$
\|F^\infty_j-\text{id}\|_{C^k(\mathbb{C}^n)}\leq \min
\left\{ \frac{d_0}{2\cdot 3^{j-1}}, \frac{c_1}{2\cdot 3^{j-1}}, \frac{d_1}{2\cdot 3^{j-2}}, 
\frac{c_2}{2\cdot 3^{j-2}},
\ldots,\frac{d_{j-1}}{2},\frac{c_j}{2} \right\}.
$$
Thus $F^\infty_1$ is a diffeomorphism satisfying (i), (ii), and (iii) of the theorem.
To see that it also satisfies (iv), notice in view of \eqref{shower} and since 
$U_{\varphi(j)}\subset K_j$ that
$$
\big[K_j\cup(X_j\setminus U_{\varphi(j)})\big]^{\widehat{}}\subset
U_j\cup(X_j\setminus U_{\varphi(j)}).
$$
%Now $F^\infty_{1}(X\setminus U_{\varphi(j)})=F^\infty_{j+1}(X_j)$ and 
Since $\|F^\infty_{j+1}-\text{id}\|_{C^k(\mathbb{C}^n)}\leq c_{j+1}/2$ 
it follows from
Corollary~\ref{LW2} that
$$
\big[K_j\cup F^\infty_{j+1}(X_j\setminus U_{\varphi(j)})\big]^{\widehat{}}\subset
U_j\cup F^\infty_{j+1}(X_j\setminus U_{\varphi(j)}).
$$
Since moreover $\|F^\infty_{j+1}-\text{id}\|_{C^k(\mathbb{C}^n)}\leq d_{j}/2$
we have
$$
K_j\cup F^\infty_{j+1}(X_j\setminus U_{\varphi(j)})=
K_j\cup F^\infty_{j+1}(X_j),\quad
U_j\cup F^\infty_{j+1}(X_j\setminus U_{\varphi(j)})=
U_j\cup F^\infty_{j+1}(X_j).
$$
%and
%$$
%U_j\cup F^\infty_{j+1}(X_j\setminus U_{\varphi(j)})=
%U_j\cup F^\infty_{j+1}(X_j).
%$$
Hence,
$$
\big[K_j\cup F^\infty_{j+1}(X_j)\big]^{\widehat{}}\subset
U_j\cup F^\infty_{j+1}(X_j)
$$
and so
$$
\big[K\cup F^\infty_1(X)\big]^{\widehat{}}\subset
U_j\cup F^\infty_{1}(X)
$$
for all $j$. Thus (iv) holds.
\end{proof}

\end{document}